\documentclass[12pt,a4paper,reqno]{amsart}
\usepackage[english]{babel}
\usepackage[applemac]{inputenc}
\usepackage[T1]{fontenc}
\usepackage{palatino}
\usepackage{amsmath}
\usepackage{amssymb}
\usepackage{amsthm}
\usepackage{amsfonts}
\usepackage{graphicx}
\usepackage[colorlinks = true, citecolor = black]{hyperref}
\author{Tuomas Orponen}
\title{A discretised projection theorem in the plane}
\address{School of Mathematics, University of Edinburgh}
\thanks{T.O. gratefully acknowledges the financial support of the Finnish foundation Jenny and Antti Wihurin Rahasto.}
\email{tuomas.orponen@helsinki.fi}

\newcommand{\R}{\mathbb{R}}
\newcommand{\N}{\mathbb{N}}

\newcommand{\calT}{\mathcal{T}}
\newcommand{\calL}{\mathcal{L}}

\numberwithin{equation}{section}

\theoremstyle{plain}
\newtheorem{thm}[equation]{Theorem}

\newtheorem{ex}[equation]{Example}

\newtheorem{proposition}[equation]{Proposition}

\newtheorem{claim}[equation]{Claim}

\theoremstyle{definition}

\newtheorem{definition}[equation]{Definition}

\theoremstyle{remark}
\newtheorem{remark}[equation]{Remark}

\begin{document}

\begin{abstract} The main result of this paper is that for any $1/2 \leq s < 2 - \sqrt{2} \approx 0.5858$, there is a number $\sigma = \sigma(s) < s$ with the following property. Let $\delta > 0$, assume that $A \subset [0,1]$ is a $(\delta,1/2)$-set, and that $E \subset [0,1]$ contains $\gtrsim \delta^{-\sigma}$ roughly $\delta^{s}$-separated points. Then there exists a number $t \in E$ such that $A + tA$ contains $\gtrsim \delta^{-s}$ $\delta$-separated points. 

For $\sigma = s$, this is essentially a consequence of Kaufman's well-known bound for exceptional sets of projections. Our proof consists of a structural observation concerning sets, for which Kaufman's bound is near-optimal, combined with (an adaptation of) Solymosi's argument for his "$4/3$" sum-product theorem. 
\end{abstract}

\maketitle

\section{Introduction} The starting point for this paper is R. Kaufman's exceptional set bound \cite{Ka} for projections from 1968. For $e \in S^{1}$, let $\pi_{e} \colon \R^{2} \to \R$ stand for the orthogonal projection $\pi_{e}(x) := e \cdot x$, and denote Hausdorff dimension by $\dim$. If $K \subset \R^{2}$ is a Borel set with $\dim K = 1$, Kaufman's bound reads
\begin{equation}\label{kaufman} \dim \{e \in S^{1} : \dim \pi_{e}(K) < s\} \leq s, \qquad 0 \leq s \leq 1. \end{equation}
The bound is sharp for $s = 1$, as shown by Kaufman and Mattila \cite{KM} in 1975, but for $s < 1$, this seems unlikely. It has been conjectured (in for instance \cite[(1.8)]{Ob}) that the sharp estimate might be
\begin{equation}\label{conjecture} \dim \{e : \dim \pi_{e}(K) < s\} \leq 2s - 1, \qquad 0 \leq s \leq 1, \end{equation}
and indeed an analogous bound is easy to verify in the discrete situation, see \eqref{sharpKaufman} below. In the continuous case, however, such an improvement appears well beyond the reach of current technology. For general sets $K$, the most notable refinement to \eqref{kaufman} is due to Bourgain \cite{Bo}, who managed to prove that
\begin{equation}\label{bourgain} \dim \{e : \dim \pi_{e}(K) < s\} \searrow 0, \qquad \text{as } s \searrow 1/2. \end{equation}
The rate of decay in \eqref{bourgain} is not explicitly stated in Bourgain's paper, and while the estimates are effective throughout, it is likely that his method of proof only gives useful information for values of $s > 1/2$ very close to $1/2$. 

At the core of Bourgain's proof, there is an estimate of the following kind. Assume that $A \subset [1,2]$ is a finite $\delta$-separated set with $|A| \sim \delta^{-1/2}$, satisfying a weak non-concentration hypothesis of the form
\begin{displaymath}  |A \cap B(x,r)| \lesssim r^{\kappa_{1}}|A|, \qquad x \in \R, \: r \geq \delta, \end{displaymath}
where $\kappa_{1} > 0$ is any positive number.\footnote{Here, $A \lesssim B$ means that $A \leq CB$ for some absolute constant $C \geq 1$, and $A \sim B$ stands for the two-sided inequality $A \lesssim B \lesssim A$.} Then, if $E \subset [0,1]$ is another finite set with $|E| > \delta^{-\kappa_{2}}$, $\kappa_{2} > 0$, satisfying a similar non-concentration hypothesis as $A$ with a constant, say, $\kappa_{3} > 0$, then one can find a point $t \in E$ such that $A + tA$ contains at least $\gtrsim \delta^{-1/2-\epsilon}$ $\delta$-separated points for some $\epsilon = \epsilon(\kappa_{1},\kappa_{2},\kappa_{3}) > 0$.

The main result of the current paper has a similar appearance, although the non-concentration assumptions are far more restrictive, and the proof technique is different. On the positive side, the number $\epsilon > 0$ is more concrete and allows us to get some information for all $s \in [1/2,2 - \sqrt{2})$, which is not a consequence of Kaufman's bound \eqref{kaufman}. We make the following definition (which is essentially copied from Katz and Tao's paper \cite{KT}):

\begin{definition} A finite set $A \subset [0,1]$ is called a $(\delta,s)$-set, if $A$ is $\delta$-separated, and
\begin{displaymath} |A \cap B(x,r)| \lesssim \left(\frac{r}{\delta}\right)^{s}, \qquad x \in \R, \: r \geq \delta.\end{displaymath}
\end{definition}
The point of the definition is that a $(\delta,s)$-set is a $\delta$-discretised model for a "continuous" set $A \subset [0,1]$ with $\mathcal{H}^{s}(A) > 0$. To demonstrate this, we quote \cite[Proposition A.1.]{FO}, although we will not need the result in this paper:
\begin{proposition}\label{FO} Let $\delta > 0$, and let $A \subset [0,1]$ be a set with $\mathcal{H}^{s}(A) =: \kappa > 0$. Then, there exists a $(\delta,s)$-set $P \subset A$ of cardinality $|P| \gtrsim \kappa \cdot \delta^{-s}$.
\end{proposition}

\begin{thm}\label{main} For $1/2 \leq s < 2 - \sqrt{2}$, there exists a number $\sigma = \sigma(s) < s$ with the following property. Assume that $\delta > 0$ is small enough, and $A \subset [0,1]$ is a $(\delta,1/2)$-set of cardinality $\sim \delta^{-1/2}$. Then, if $E \subset [0,1]$ contains $\gtrsim \delta^{-\sigma}$ points, which are $\gtrsim \delta^{-s}$-separated, there exists $t \in E$ such that $A + tA$ contains $\gtrsim \delta^{-s}$ $\delta$-separated points.
\end{thm}

\begin{remark} For comparison, let us remark that any standard proof of Kaufman's bound \eqref{kaufman} combined with Proposition \ref{FO} could be used to deduce Theorem \ref{main} with $\sigma = s$ (at least, if one allows the constants in the $\lesssim$-notation to be of the order $\log(1/\delta)$). So, the whole point of Theorem \ref{main} is to find $\sigma$ slightly smaller than $s$. \end{remark}


The proof strategy in Theorem \ref{main} is roughly the following. If no such $\sigma(s) < s$ existed, then for any $\sigma < s$ we could find a set $A$ as in the hypotheses, such that $A + tA$ would contain $\ll \delta^{-s}$ points for all $t \in E$. In particular, we assume that $1 \in E$. As $\sigma \nearrow s$, such a set $A$ -- or rather $A \times A$ -- gets almost as bad as a set can be in view of Kaufman's bound \eqref{kaufman}. So, whatever estimates are used in the proof of that bound must be quite sharp for this particular $A$. It turns out that one can relatively easily extract some structural information about $A \times A$ based on this knowledge. More precisely, there exist many "fan" structures, where a significant part of $A \times A$ is concentrated on relatively few $\delta$-tubes emanating from a fixed point in $A \times A$. 

To exploit the existence of fans, we borrow and adapt an argument of J. Solymosi \cite{So} in discrete sum-product theory. His theorem says that $\min \{|A + A|, |A \cdot A|\} \gtrsim |A|^{4/3}/\log |A|$ for any finite set $A \subset \R$, so the result in itself is not applicable in our situation. However, the same sort of "fan" structure arises from the assumption that $|A \cdot A|$ is small, so we may exploit Solymosi's idea to conclude that $A + A$ has size at least $\delta^{-1/2 - c}$ for some explicit constant $c > 0$. Assuming that $s < 1/2 + c$ leads to a contradiction, because $1 \in E$. Curiously, the passage from points and lines in Solymosi's proof to $\delta$-discs and tubes in our setting does not cause serious trouble here, thanks to the strong non-concentration assumptions. 

The proof of Theorem \ref{main} can be found in Section \ref{mainProof}. Before that, in the next section, we briefly discuss similar -- and much easier -- problems in the discrete situation. 

\section{Remarks in the discrete situation}

The discrete analogue of Kaufman's bound \eqref{kaufman} would be the following. Let $P \subset \R^{2}$ be any set of $n$ points. Then
\begin{equation}\label{discreteKaufman} |\{e : |\pi_{e}(P)| \leq n^{s}\}| \lesssim n^{s}, \qquad 0 \leq s < 1. \end{equation} 
Why this is the natural analogue is not a matter of pure numerology, but the proof is essentially the same $L^{2}$ type argument that gives \eqref{kaufman}. It can be sketched in a few words as follows. 
\begin{proof}[Proof of \eqref{discreteKaufman}] Fix $0 \leq s < 1$. Then, for any \emph{bad} $e \in S^{1}$ such that $|\pi_{e}(P)| \leq n^{s}$ there must correspond many pairs $(p,q) \in P \times P$, $p \neq q$, such that $\pi_{e}(p) = \pi_{e}(q)$. In fact, one can easily check that the number of such pairs is at least of the order $n^{2 - s}$ (this numerology would arise, if the points in $P$ were equally divided on $n^{s}$ lines of the form $\pi_{e}^{-1}\{t\}$ with $t \in \pi_{e}(P)$, and Cauchy-Schwarz can be used to show that this is the worst case, indeed). For various bad $e \in S^{1}$, the $\gtrsim n^{2 - s}$ pairs $(p,q)$ are distinct, because the directions $(p - q)/|p - q| \perp e$ are. But there are only $\sim n^{2}$ pairs in $P \times P$, so there can be at most $\lesssim n^{s}$ bad directions $e$. \end{proof}

It is known that \eqref{discreteKaufman} is not sharp for any $s < 1$, and in fact the sharp bound is known, too, see \cite[Proposition 1.8]{Or}:
\begin{equation}\label{sharpKaufman} |\{e : |\pi_{e}(P)| \leq n^{s}\}| \lesssim n^{2s - 1}, \qquad 1/2 \leq s < 1. \end{equation}
The proof of \eqref{sharpKaufman} is very short, but unfortunately relies heavily on the Szemer\'edi-Trotter bound and so does not easily lend itself easily for continuous problems. Now, the starting point of this paper was the following: even if we already know that \eqref{discreteKaufman} is not sharp, we can still consider the hypothetical case that it were, and see what the proof above tells us about the (non-existent) extremisers. Perhaps this information could also be exploited in the continuous case?

With this question in mind, assume that $\sigma < s$ is very close to $s$, and $P \subset \R^{2}$ is an $n$-point set such that $|\{e : |\pi_{e}(P)| \leq n^{s}\}| \sim n^{\sigma}$. Then, a quick re-examination of the proof of \eqref{discreteKaufman} reveals that there exists a set $G \subset P \times P$ of cardinality $|G| \gtrsim n^{2 - (s - \sigma)}$ such that the pairs in $G$ only span $\sim n^{\sigma}$ directions. In symbols,
\begin{displaymath} S(G) := \left\{\frac{p - q}{|p - q|} : (p,q) \in G\right\} \end{displaymath}
satisfies $|S(G)| \sim n^{\sigma}$. Does this tell us something about the set $P$?

In the discrete case, the answer is strongly positive. The only previously available result in the right direction appears to be due to P. Ungar from 1982, saying that if $|S(P \times P)| < n - 1$, then the whole set $P$ is contained on a single line. For us, this would be a great conclusion, because a set contained on a single line has $\lesssim 1$ bad directions in terms of projections. However, we only have information available about a major subset $G \subset P \times P$, so there is no hope to conclude something quite as strong. In our situation, the optimal result would be to show that "small $S(G)$ for a large $G$ forces many points of $P$ on a single line." Indeed, such a result follows from the Szemer\'edi-Trotter bound:
\begin{thm}\label{mainDiscrete} Let $P \subset \R^{2}$ be a finite set with $|P| = n$. Let $s \geq 1/2$, and assume that $G \subset P \times P$ is a set of pairs of cardinality $|G| \geq n^{1 + s}$ such that
\begin{displaymath} |S(G)| \leq \frac{cn^{2s - 1}}{\log^{12} n} \end{displaymath}
for a suitable absolute constant $c > 0$. Then $|\ell \cap P| \gtrsim n^{s}/\log^{4} n$ for some line $\ell \subset \R^{2}$. This result is sharp in the sense that assuming $|S(G)| \leq n^{t}$ for any $t < 2s - 1$ does not improve the conclusion, and on the other hand the slightly weaker assumption $|S(G)| \lesssim n^{2s - 1}$ could, at best, yield the conclusion $|\ell \cap P| \gtrsim n^{1/2}$ for some line $\ell$. 
\end{thm}

\begin{remark} We postpone the proof to the appendix, because the technique does not help us tackle continuous problems. However, we wish to point out that as $s \nearrow 1$, Theorem \ref{mainDiscrete} essentially recovers Ungar's result (at least as far as the exponents are concerned), and the bound \eqref{sharpKaufman} is also a corollary (mod logarithmic factors). To see the latter claim, assume that \eqref{sharpKaufman} fails for some $1/2 < s < 1$, and find many pairs spanning few directions as in the proof of the bound \eqref{discreteKaufman}. Then, Theorem \ref{mainDiscrete} produces a line containing so many points of $P$ that a contradiction is reached (such a set $P$ cannot have more than two "bad" directions $e$). \end{remark}

In the continuous situation, unfortunately, we could not find a way to implement the strategy outlined above; in other words, the proof of Theorem \ref{main} does start by using the hypothetical sharpness of Kaufman's bound to find many pairs spanning only few directions, just as above, but without Szemer\'edi-Trotter at our disposal, we do not know how to use this to extract sufficiently many $\delta$-discs in the vicinity of a single $\delta$-tube. It remains open, whether this strategy has any potential in the continuous case. 

Instead, the "many pairs spanning few directions" information will be simply used to find a \emph{fan}. This refers to a (rather small) collection of tubes, all containing a common point, the union of which contains a significant part of the points in $P$ -- or $A \times A$, as in Theorem \ref{main}. In the discrete case, finding a fan is straightforward: if there are $\gtrsim n^{2 - (s - \sigma)}$ pairs in $P \times P$ spanning only $\sim n^{\sigma}$ directions, then the generic point in $P$ can be connected to $\gtrsim n^{1 - (s - \sigma)}$ other points of $P$ using $\lesssim n^{\sigma}$ lines. The union of these lines forms a fan. 

In the continuous situation, Theorem \ref{main}, the idea is essentially the same, but the existence of a fan structure requires a stronger non-concentration assumption on the bad set of directions $E$ than simple $\delta$-separation. 
\begin{figure}[h!]
\begin{center}
\includegraphics[scale = 0.7]{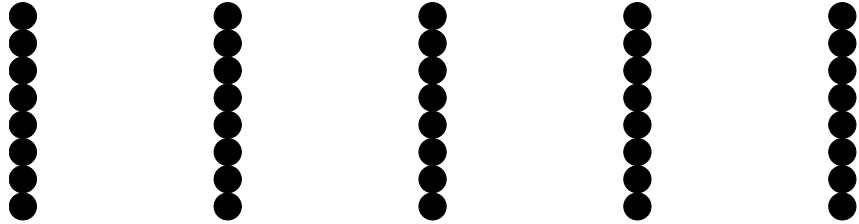}
\caption{An example to keep in mind: $\sim \delta^{-1/2}$ columns, each consisting of $\sim \delta^{-1/2}$ adjacent discs of radius $\delta$.}\label{fig3}
\end{center}
\end{figure}  
For instance, the standard counterexample in Figure \ref{fig3} portrays a $\delta$-discretised $1$-dimensional set, which projects into an essentially $1/2$-dimensional $\delta$-discretised set in $\sim \delta^{-1/2}$ $\delta$-separated directions around the vertical one. Nonetheless, there are no fans to be seen, and the reason is precisely the lack of separation for the directions of small projection.


\section{Proof of the main theorem}\label{mainProof}

\subsection{Finding a fan}\label{findingFan} As announced, the first step of the proof is to use the hypothetical (almost) sharpness of Kaufman's bound to find a fan structure, which is defined in Claim \ref{fanClaim} below. For this purpose, we can make do with weaker assumptions than those in Theorem \ref{main}: we consider a finite set $B \subset B(0,2)$, consisting of $\sim \delta^{-1}$ $\delta$-separated points and satisfying the $(\delta,1)$-set inequality
\begin{equation}\label{nonConB} |B \cap B(x,r)| \lesssim \frac{r}{\delta}, \qquad x \in \R^{2}, \: r \geq \delta. \end{equation}
We will later apply the conclusions to $B = A \times A$, which clearly satisfies \eqref{nonConB}. Moreover, for this subsection, it suffices to assume that the angles in $E$ satisfy the $(\delta,s)$-set criterion,
\begin{equation}\label{nonConE} |E \cap B(x,r)| \lesssim \left(\frac{r}{\delta}\right)^{s}, \qquad x \in S^{1}, \: \delta \leq r \leq 1, \end{equation}
which is weaker than $\delta^{s}$-separation. Finally, for convenience, we will assume that $E$ is a subset of $S^{1}$ instead of $\R$, and that $e_{0} := \tfrac{1}{\sqrt{2}}(1,1) \in E$.\footnote{This assumption will eventually be used to control the size of $\pi_{e_{0}}(A \times A) = A + A$. If $e_{0} \notin E$, the proof runs the same way by considering $A \times xA$, where we have good control for $A + xA$.} 

\begin{definition} A tube $T \subset \R^{2}$ of width $2\delta$ is $\sigma$-\emph{good}, if for some absolute constant $c \geq 1$,
\begin{displaymath} \mathop{\sum_{x,y \in B \cap T}}_{x \neq y} |x - y|^{s - 1} \lesssim \delta^{2(s - 1) + c(\sigma - s)}\log^{3/2}\left(\frac{1}{\delta}\right). \end{displaymath} 
\end{definition}

\begin{claim}\label{fanClaim} Assume that $E$ contains $\gtrsim \delta^{-\sigma}$ directions for some $0 < \sigma < s$ and satisfies \eqref{nonConE}. Also, assume that for each $e \in E$, the projection $\pi_{e}(B)$ only contains $\lesssim \delta^{-s}$ $\delta$-separated points. Then, there exists a constant $\tau = \tau(\sigma,s) > 0$ and a point $b_{0} \in B$ with the following property. There is a subset $E' \subset E$, and for each $e \in E'$ a $\sigma$-good $2\delta$-tube $T_{e}(b_{0})$ containing $b_{0}$ such that
\begin{equation}\label{form1} \left| B \cap \bigcup_{e \in E'} T_{e}(b_{0})\right| \gtrsim \delta^{\tau - 1}. \end{equation}
Furthermore, $\tau(\sigma,s) \to 0$, as $\sigma \to s$.
\end{claim} 

\begin{proof}[Proof of claim] The first step is to find plenty of good tubes to play with. For each $e \in E$, partition $\R^{2}$ into tubes of width $2\delta$ perpendicular to $e$. Then, fix $N \sim \delta^{-s}$ and, for each $e \in E$, pick exactly $N$ of these tubes $T_{e,1},\ldots,T_{e,N}$ such that 
\begin{displaymath} B \subset \bigcup_{n = 1}^{N} T_{e,n}, \qquad e \in E. \end{displaymath}
This is possible by the assumption on $\pi_{e}(B)$. The next estimation shows that, on average, the tubes $T_{e,n}$ are $\sigma$-good:
\begin{align*} \frac{1}{|E|N} \sum_{e \in E} \sum_{n = 1}^{N}  \mathop{\sum_{x,y \in B \cap T_{e,n}}}_{x \neq y} |x - y|^{s - 1} & = \frac{1}{|E|N} \mathop{\sum_{x,y \in B}}_{x \neq y} \sum_{e \in E} \sum_{n = 1}^{N} |x - y|^{s - 1} \chi_{(x,y) \in T_{e,n}^{2}}(x,y)\\
& \lesssim \frac{1}{|E|N} \mathop{\sum_{x,y \in B}}_{x \neq y} |x - y|^{-1} \lesssim \frac{\delta^{-2}}{|E|N}\log\left(\frac{1}{\delta}\right)\\
& \sim \delta^{\sigma + s - 2}\log\left(\frac{1}{\delta}\right) = \delta^{2(s - 1) + (\sigma - s)}\log\left(\frac{1}{\delta}\right).  \end{align*}
Pick a large constant $C > 0$, and discard all the directions $e \in E$ such that 
\begin{displaymath} \frac{1}{N} \sum_{n = 1}^{N} \mathop{\sum_{x,y \in B \cap T_{e,n}}}_{x \neq y} |x - y|^{s - 1} \geq C\delta^{\sigma + s - 2}\log\left(\frac{1}{\delta}\right). \end{displaymath}
The remaining directions are called $E_{0}$, and the preceding computation shows that $|E_{0}| \gtrsim |E| \sim \delta^{-s}$ for large enough $C$. Next, for all $e \in E_{0}$, discard the tubes $T_{e,n}$ such that
\begin{displaymath} \mathop{\sum_{x,y \in B \cap T_{e,n}}}_{x \neq y} |x - y|^{s - 1} \geq D\delta^{\sigma + s - 2}\log\left(\frac{1}{\delta}\right), \end{displaymath} 
where $D \sim \sqrt{\log(1/\delta)}\delta^{(\sigma - s)/2}$ is a constant to be determined shortly. We wish to make sure that not too many points of $B$ lie in the discarded tubes. By definition of $E_{0}$, the number of the discarded tubes, the family of which is denoted by $\calT_{e}^{b}$, is bounded by $|\calT_{e}^{b}| \leq (C/D)N$. Combining this fact, plus the definition of $E_{0}$, and the Cauchy-Schwarz estimate 
\begin{displaymath} \frac{1}{N} \sum_{n = 1}^{N} \mathop{\sum_{x,y \in B \cap T_{e,n}}}_{x \neq y} |x - y|^{s - 1} \gtrsim \frac{1}{N} \sum_{T \in \calT_{e}^{b}} |B \cap T_{e,n}|^{2} \gtrsim \frac{1}{(C/D)^{2}N^{2}} \left( \sum_{T \in \calT_{e}^{b}} |B \cap T_{e,n}| \right)^{2} \end{displaymath} 
shows that
\begin{displaymath} \left| B \cap \bigcup_{T \in \calT_{e}^{b}} T \right| \lesssim \sqrt{C\log(1/\delta)\delta^{\sigma + s - 2}} \cdot (C/D)N \sim \frac{C^{3/2}}{D}\delta^{-1 + (\sigma - s)/2}\sqrt{\log(1/\delta)}. \end{displaymath}
In other words, since $|B| \sim \delta^{-1}$, one can choose $D \sim C^{3/2}\sqrt{\log(1/\delta)}\delta^{(\sigma - s)/2}$ so large that the following holds: for each good direction $e \in E_{0}$, the points of $B$ covered by the "bad" tubes in $\calT_{e}^{b}$ only constitute a $1/10$ of all the points in $B$; hence, if $\calT_{e} := \{T_{e,1},\ldots,T_{e,N}\} \setminus \calT_{e}^{b}$, we have
\begin{equation}\label{form8} \left| B \cap \bigcup_{T \in \calT_{e}} T \right| \gtrsim |B| \sim \delta^{-1}, \end{equation}
and for all $T \in \calT_{e}$
\begin{displaymath}  \mathop{\sum_{x,y \in B \cap T}}_{x \neq y} |x - y|^{s - 1} \leq D\delta^{\sigma + s - 2}\log\left(\frac{1}{\delta}\right) \lesssim \delta^{2(s - 1) + 2(\sigma - s)}\log^{3/2}\left(\frac{1}{\delta}\right), \end{displaymath}
which means that the tubes in $\calT_{e}$ are $\sigma$-good.

Next, we start looking for a "fan" structure. Given $b,b' \in B$, $b \neq b'$, and $e \in E_{0}$, let $b \sim_{e} b'$ stand for the relation of $b$ and $b'$ sharing a tube in $\calT_{e}$. First, we make a quick calculation concerning the number of pairs $(b,b')$ such that $b \sim_{e} b'$ for some $e \in E_{0}$. Assume that every tube in $\calT_{e}$ contains at least two points of $B$ (if this is not the case, discard the single-point tubes, and use $|\calT_{e}| \leq N \sim \delta^{-s}$ to conclude that \eqref{form8} still holds for the remaining family of tubes). Then, from Cauchy-Schwarz and \eqref{form8}, 
\begin{align*} |\{(b,b') : b \sim_{e} b'\}| & \gtrsim \sum_{T \in \calT_{e}} |B \cap T|^{2}\\
& \geq \frac{1}{|\calT_{e}|} \left(\sum_{T \in \calT_{e}} |B \cap T| \right)^{2} \gtrsim \delta^{s - 2}, \end{align*} 
and so, recalling that $|E_{0}| \gtrsim \delta^{-\sigma}$,  
\begin{equation}\label{form2} \sum_{b, b' \in B} |\{e \in E_{0} : b \sim_{e} b'\}| = \sum_{e \in E_{0}} |\{(b,b') : b \sim_{e} b'\}| \gtrsim \delta^{s - \sigma - 2}. \end{equation}

Assume that the conclusion of our claim, namely \eqref{form1}, fails for every $b_{0} = b \in B$, and for some
\begin{equation}\label{form3} \tau > \frac{s - \sigma}{1 - s}. \end{equation}
Fixing $b \in B$, this information can be used to bound the quantity
\begin{displaymath} \sum_{b' \in B} |\{e \in E_{0} : b \sim_{e} b'\}| \end{displaymath} 
from above as follows. First, the non-concentration condition \eqref{nonConE} for $E$, thus $E_{0}$, gives the universal bound
\begin{equation}\label{form7} |\{e \in E_{0} : b \sim_{e} b'\}| \lesssim \frac{1}{|b - b'|^{s}}, \end{equation}
simply because the set of possible directions $e \in S^{1}$ such that $b \sim_{e} b'$ are contained in two arcs of length $\lesssim \delta/|b - b'|$. For $e \in E_{0}$ and $b \in B$, let $T_{e}(b)$ be the unique tube $T_{e,n}$ containing $b$, if $T_{e,n} \in \calT_{e}$, and let $T_{e}(b) = \emptyset$ otherwise. If
\begin{displaymath} b' \notin B \cap \bigcup_{e \in E_{0}} T_{e}(b) =: G(b), \end{displaymath}
then simply
\begin{displaymath} \{e \in E_{0} : b \sim_{e} b'\} = \emptyset. \end{displaymath}
Hence,
\begin{displaymath} \sum_{b,b'} |\{e \in E_{0} : b \sim_{e} b'\}| \lesssim \sum_{b \in B} \sum_{b' \in G(b)} \frac{1}{|b - b'|^{s}}. \end{displaymath} 
For $\delta \leq 2^{j} \leq 1$, consider the sets $A_{j}(b) := \{b' \in G(b) : |b - b'| \sim 2^{j}\}$, for which the non-concentration inequality \eqref{nonConB} gives the bound
\begin{displaymath} |A_{j}(b)| \lesssim \min\left\{\frac{2^{j}}{\delta},\delta^{\tau - 1}\right\} \leq \left(\frac{2^{j}}{\delta}\right)^{s} \cdot \delta^{(\tau - 1)(1 - s)}. \end{displaymath}
Plugging this into the previous equation gives
\begin{align*} \sum_{b,b'} |\{e \in E_{0} : b \sim_{e} b'\}| & \lesssim \sum_{b \in B} \sum_{\delta \leq 2^{j} \leq 1} \sum_{b' \in A_{j}(B)} \frac{1}{|b - b'|^{s}}\\
& \lesssim \sum_{b \in B} \sum_{\delta \leq 2^{j} \leq 1} 2^{-js} \left(\frac{2^{j}}{\delta}\right)^{s} \cdot \delta^{(\tau - 1)(1 - s)}\\
& \sim |B| \cdot \log \left(\frac{1}{\delta}\right) \cdot \delta^{-1 + \tau(1 - s)} \sim \delta^{-2 + \tau(1 - s)} \cdot \log \left(\frac{1}{\delta}\right). \end{align*} 
This upper bound contradicts a combination of \eqref{form2} and \eqref{form3}. The claim is hence established for any $\tau(\sigma,s) > (s - \sigma)/(1 - s)$. In other words, we have found a point $b_{0} \in B$ such that $|G(b_{0})| \gtrsim \delta^{\tau - 1}$. Finally, for this particular $b_{0}$, let $E' \subset E_{0}$ consist of those directions such that $T_{e}(b_{0}) \neq \emptyset$ (which implies that $T_{e}(b_{0}) \in \calT_{e}$ for $e \in E'$).
\end{proof}

\subsection{Using the fan} As stated, we apply the claim with $B = A \times A$. In the second -- and final -- phase of the proof, we will need the full strength of the hypotheses. The key idea is borrowed from J. Solymosi \cite{So}: if a product set of the form $A \times A$ is (to a large extent) contained in a "fan" of thin tubes, such as $\bigcup T_{e}(b_{0})$, then the sumset $A + A$ has quantitatively larger size than $A$.

To prove this, we assume, without loss of generality, that $E' = E$, that the directions in $E$ lie in the northeast quartile $Q$ of $\R^{2}$, and also that
\begin{equation}\label{form10} \left| (A \times A) \cap \bigcup_{e \in E} T_{e}(b_{0}) \cap (b_{0} + Q) \right| \gtrsim \delta^{\tau - 1}. \end{equation} 
In English, a significant part of the points in $(A \times A) \cap \bigcup T_{e}(b_{0})$ is contained to the northeast of $b_{0}$. The rest of the points in $A \times A$ are not needed, so we assume that $(A \times A) \cap T_{e}(b_{0})$ is contained in the quartile $b_{0} + Q$ for every $e \in E$. 

Write 
\begin{displaymath} \kappa := \kappa(\sigma,s) := \max\{\tau(\sigma,s), c(s - \sigma)\}. \end{displaymath}
For each tube $e \in E'$, discard all the points $x \in (A \times A) \cap T_{e}(b_{0})$ such that
\begin{equation}\label{form11} \mathop{\sum_{y \in (A \times A) \cap T_{e}(b_{0})}}_{x \neq y} |x - y|^{s - 1} \geq C\delta^{s - 1 - 2\kappa}\log^{3/2}\left(\frac{1}{\delta}\right) \end{equation}
for a large constant $C > 0$. Since $T_{e}(b_{0})$ is a good tube, the number of such points is bounded by $\lesssim \delta^{s - 1 + \tau}/C$, and since the number of tubes $T_{e}(b_{0})$, $e \in E$, is bounded by $\lesssim \delta^{-s}$, we find that only $\lesssim \delta^{\tau - 1}/C$ points $x \in A \times A$ were discarded altogether. Hence, choosing $C > 0$ large enough, \eqref{form10} continues to hold with $A \times A$ replaced by the set of remaining points, and moreover we can now assume that the inequality opposite to \eqref{form11} holds for all $x \in A \times A$ and $e \in E$. We will use this fact in the slightly weaker form
\begin{equation}\label{form12} |(A \times A) \cap T_{e}(b_{0}) \cap B(x,r)| \lesssim \delta^{-2\kappa}\log^{3/2}\left(\frac{1}{\delta}\right) \cdot \left(\frac{r}{\delta}\right)^{1 - s}, \quad x \in \R^{2},\; r \geq \delta.  \end{equation}
We also need to arrange so that there are no points of $A \times A$ too close to $b_{0}$: we throw away all the points $x \in A \times A$ that lie in a ball $B(b_{0},C\delta^{1 - s})$, where $C > 0$ is an absolute constant large enough for future purposes. This removal procedure only costs $\lesssim C\delta^{1 - s}/\delta = C\delta^{-s}$ points according to the non-concentration inequality for $A$, so \eqref{form10} remains valid for small enough $\delta > 0$.

Observe that \eqref{form12} combined with the assumption $A \times A \subset B(0,1)$ implies that any tube $T_{e}(b_{0})$, $e \in E$, can contain at most $\lesssim \delta^{-2\kappa}\log^{3/2}(1/\delta) \cdot \delta^{s - 1}$ points of $A \times A$. Combined with \eqref{form10} and $|E| \lesssim \delta^{-s}$, this means that there exists a subcollection $\calT$ of the tubes $T_{e}(b_{0})$ of cardinality 
\begin{equation}\label{form14} |\calT| \gtrsim \delta^{3\kappa}\log^{-3/2}\left(\frac{1}{\delta}\right) \cdot \delta^{-s} \gtrsim \delta^{4\kappa - s} \end{equation}
such that
\begin{equation}\label{form13} |(A \times A) \cap T| \gtrsim \delta^{\kappa + s - 1}, \qquad T \in \calT. \end{equation} 
In the sequel, only these tubes will be of interest, and we enumerate them, $\calT = \{T_{1},\ldots,T_{N}\}$, so that tubes with consecutive indices correspond to consecutive directions in $E$.



\begin{figure}[h!]
\begin{center}
\includegraphics[scale = 0.6]{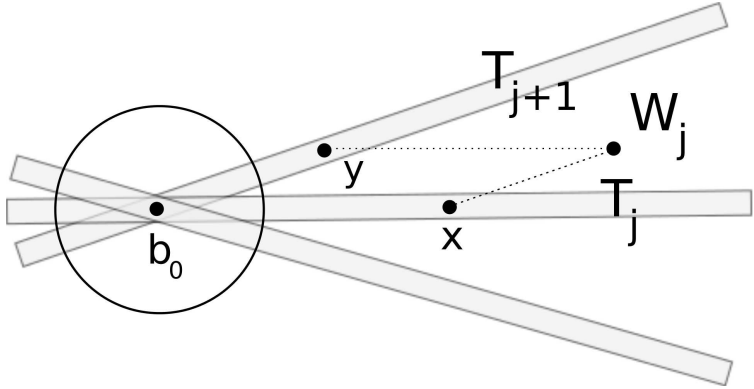}
\caption{The vector sum $(x + y) - b_{0}$ in the white conical region $W_{j}$.}\label{fig1}
\end{center}
\end{figure}

The next, and final, step of the proof is to consider vector sums of the form
\begin{equation}\label{form5} b_{0} + (x - b_{0}) + (y - b_{0}) = (x + y) - b_{0}, \end{equation}
where $x \in (A \times A) \cap T_{j}$ and $y \in (A \times A) \cap T_{j + 1}$ for some $1 \leq j \leq N - 1$, see Figure \ref{fig1}. Such vector sums are always contained in the conical region spanned by two boundary lines of the tubes $T_{j}$ and $T_{j + 1}$. \emph{A fortiori}, if $x,y$ lie outside the ball $B(b_{0},C\delta^{1 - s})$ for large enough $C$, then elementary geometry and the $\delta^{s}$-separation of the directions $e \in E$ gives that the sum in \eqref{form5} will be contained in the \emph{white conical region $W_{j}$} outside the tubes $T_{e}(b_{0})$, depicted in Figure \ref{fig1}. The regions $W_{j}$ are disjoint for distinct pairs of consecutive tubes, so the plan is to find many $\delta$-separated vector sums in each $W_{j}$, and then multiply by their number $|W_{j}| = N - 1$. By definition, the vector sums of the form \eqref{form5} are contained in the set
\begin{displaymath} (A \times A) + (A \times A) - b_{0} = (A + A) \times (A + A) - b_{0}, \end{displaymath}
so finding many $\delta$-separated vector sums will result in a lower bound for the number of $\delta$-separated points in $(A + A) \times (A + A)$, and hence $A + A$.


Now, we start forming the vector sums from pairs of points in $(A \times A) \cap T_{j}$ and $(A \times A) \cap T_{j + 1}$. We wish to find $\delta$-separated such sums, so we need to compute, how many vector sums can land within a distance $\delta$ of each other. Here the separation of the directions in $E$ is essential.
\begin{figure}[h!]
\begin{center}
\includegraphics[scale = 0.7]{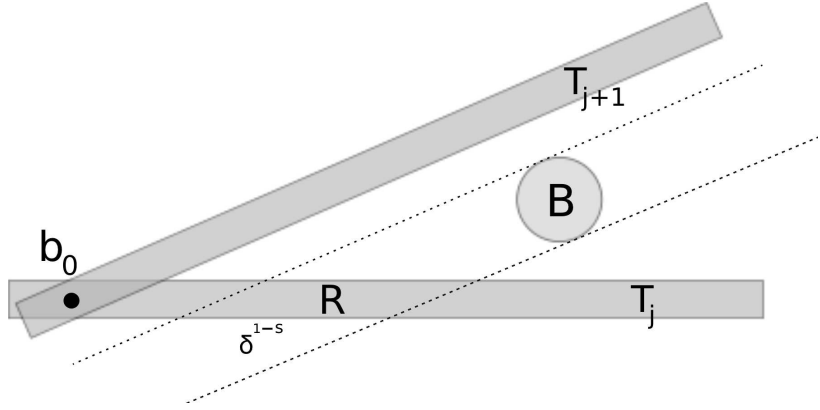}
\caption{In order for the sum $(x + y) - b_{0}$ to hit an arbitrary $\delta$-ball $B$, the point $x \in T_{j}$ needs to lie in the region $R \subset T_{j}$ of width $\sim \delta^{1 - s}$.}
\end{center}
\end{figure}
Given an arbitrary $\delta$-ball $B \subset W_{j}$, it is only possible for the sum $(x + y) - b_{0}$ to hit $B$, if $x \in T_{j}$, say, is chosen inside a rectangle $R \subset T_{j}$ of dimensions $\sim \delta^{1 - s} \times \delta$. For any such $x$, there are only $\lesssim 1$ choices of $y \in (A \times A) \cap T_{j}$ such that $(x + y) - b_{0} \in B$. Applying the non-concentration inequality \eqref{form12}, we may conclude that there exist at most 
\begin{displaymath} \lesssim \delta^{-2\kappa}\log^{3/2}\left(\frac{1}{\delta}\right)\cdot \left(\frac{\delta^{1 - s}}{\delta}\right)^{1 - s} \lesssim \delta^{s(s - 1) - 3\kappa} \end{displaymath}
pairs $(x,y)$ such that $x \in (A \times A) \cap T_{j}$, $y \in (A \times A) \cap T_{j + 1}$, and $(x + y) - b_{0} \in B$.

The conclusion above holds for an arbitrary $\delta$-ball $B \subset W_{j}$, so the we may estimate the number of $\delta$-separated sums $(x + y) - b_{0} \in W_{j}$ from below by
\begin{displaymath} \gtrsim \delta^{s(1 - s) + 3\kappa} \cdot |(A \times A) \cap T_{j}| \cdot |(A \times A) \cap T_{j + 1}|  \gtrsim \delta^{s(1 - s) + 5\kappa + 2(s - 1)} = \delta^{5\kappa + (2 - s)(s - 1)}, \end{displaymath}
the second inequality being the content of \eqref{form13}. Summing up the contributions from the various regions $W_{j}$, $1 \leq j \leq N - 1$, and recalling that $N \gtrsim \delta^{4\kappa - s}$ by \eqref{form14}, we obtain altogether
\begin{displaymath} \gtrsim \delta^{9\kappa + (2 - s)(s - 1) - s} = \delta^{9\kappa + 2s - s^{2} - 2} \end{displaymath}
$\delta$-separated sums of the form $(x + y) - b_{0}$. As we observed earlier, this means that the set $(A + A) \times (A + A) - b_{0}$ contains at least this many $\delta$-separated points, and so $A + A$ contains at least
\begin{displaymath} \gtrsim \delta^{5\kappa + s - s^{2}/2 - 1} \end{displaymath}
$\delta$-separated points. Finally recall that our starting assumption was that $A + A$ only contains $\lesssim \delta^{-s}$ $\delta$-separated points. This forces
\begin{displaymath} \delta^{5\kappa + s - s^{2}/2 - 1} \lesssim \delta^{-s}, \end{displaymath}
which in the range $s < 2 - \sqrt{2}$ this gives a lower bound for $\kappa$ -- and hence an upper bound for $\sigma < s$, because $\kappa \to 0$, as $\sigma \to s$. The proof of Theorem \ref{main} is complete.

\appendix

\section{Finding many points on a line}

This final section contains the proof of Theorem \ref{mainDiscrete}, repeated below:
\begin{thm} Let $P \subset \R^{2}$ be a finite set with $|P| = n$. Let $s \geq 1/2$, and assume that $G \subset P \times P$ is a set of pairs of cardinality $|G| \geq n^{1 + s}$ such that
\begin{displaymath} |S(G)| \leq \frac{cn^{2s - 1}}{\log^{12} n} \end{displaymath}
for a suitable absolute constant $c > 0$. Then $|\ell \cap P| \gtrsim n^{s}/\log^{4} n$ for some line $\ell \subset \R^{2}$. This result is sharp in the sense that assuming $|S(G)| \leq n^{t}$ for any $t < 2s - 1$ does not improve the conclusion, and on the other hand the slightly weaker assumption $|S(G)| \lesssim n^{2s - 1}$ could, at best, yield the conclusion $|\ell \cap P| \gtrsim n^{1/2}$ for some line $\ell$. 
\end{thm} 

We establish the sharpness claims first through two examples; this will hopefully illustrate, where the numerology in the exponents comes from. The first example, a simple grid with $n$ points, shows that the assumption $|S(G)| \lesssim n^{2s - 1}$ is insufficient for the desired conclusion:

\begin{ex}\label{ex1} Let $P = \{0,\ldots,n^{1/2}\} \times \{0,\ldots,n^{1/2}\}$. Write $r := n^{s - 1/2}$, and let $P_{r} := \{p \in P : |p| \leq r\}$. Then $|P_{r}| \sim r^{2}$, and the various pairs of points in $P_{r}$ span $\sim r^{2}$ slopes. As one easily checks, for each of these slopes $e$, there are $\gtrsim n^{3/2}/r$ pairs $(p,q) \in P \times P$ such that $s(p,q) = e$. So, we can find $\gtrsim r^{2} \cdot n^{3/2}/r = r \cdot n^{3/2} = n^{1 + s}$ pairs spanning $\sim r^{2} = n^{2s - 1}$ slopes. In particular, the assumptions $|G| \gtrsim n^{1 + s}$ and $S(G) \lesssim n^{2s - 1}$ \textbf{do not} guarantee a line $\ell$ with $|\ell \cap P| \gg n^{1/2}$.
\end{ex}

The second example shows, that the strengthening the assumption to $|S(G)| \leq n^{t}$, $t < 2s - 1$, is useless in view of the conclusion:

\begin{ex}\label{ex2} Fix $1/2 < s < 1$, draw $k \sim n^{1 - s}$ parallel lines, with slope $e$, say, and place $n/k \sim n^{s}$ points on each line. Call these points $P$. Then $|P| = n$, and there are $(n/k)^{2} \cdot k = n^{2}/k \sim n^{1 + s}$ pairs in $(p,q) \in P \times P$ with $s(p,q) = e$. If $G$ is the set of those pairs, we have $|G| \gtrsim n^{1 + s}$ and $S(G) = 1 \lesssim n^{t}$ \textbf{for any} $t > 0$.
\end{ex}

Finally, we prove Theorem \ref{mainDiscrete}:

\begin{proof}[Proof of Theorem \ref{mainDiscrete}] By assumption, slopes of the pairs in $G$ are contained in a set $E \subset S^{1}$ of cardinality $|E| \ll n^{2s - 1}/\log^{12} n$. For $j \geq 0$, let $E_{j} \subset E$ be the subset of slopes $e$, for which there exist between $2^{j}$ and $2^{j + 1}$ pairs $(p,q) \in G$ with $s(p,q) = e$. Then
\begin{displaymath} n^{1 + s} \leq |G| \sim \sum_{j = 0}^{\infty} |E_{j}| \cdot 2^{j}, \end{displaymath}
so there exists $j \geq 0$ with
\begin{equation}\label{form20} \frac{n^{1 + s}}{2^{j} \cdot j^{2}} \lesssim |E_{j}| \lesssim \frac{n^{1 + s}}{2^{j}}. \end{equation}
Moreover, $j$ can be chosen rather large. Namely, for $T \in \N$, 
\begin{displaymath} \sum_{j = 0}^{T} |E_{j}| \cdot 2^{j} \leq \sum_{j = 0}^{T} |E| \cdot 2^{j} \leq \frac{cn^{2s - 1}}{\log^{12} n} \sum_{n = 0}^{T} 2^{j} \sim \frac{cn^{2s - 1}}{\log^{12} n} \cdot 2^{T}, \end{displaymath}
and this is $\ll |G| = n^{1 + s}$, as long as $2^{T} \ll c^{-1}n^{2 - s} \cdot \log^{12} n$. Consequently, $j$ can be chosen satisfying \eqref{form20} and
\begin{equation}\label{form21} 2^{j} \geq Cn^{2 - s} \cdot \log^{12} n, \end{equation}
where $C$ can be made large by assuming that $c$ is small. 

Fix $j$ with these good properties. Then, for each $e \in E_{j}$, let $\calL_{e,k}$, $k \geq 0$, be the set of lines with slope $e$, which contain between $2^{k}$ and $2^{k + 1}$ points of $P$. Since there are $\sim 2^{j}$ pairs of $G \subset P \times P$ contained on such lines, one has
\begin{displaymath} 2^{j} \lesssim \sum_{k = 0}^{\infty} |\calL_{e,k}| \cdot 2^{2k}. \end{displaymath}
Consequently, there exists $k \geq 0$ satisfying
\begin{equation}\label{form22} |\calL_{e,k}| \gtrsim \frac{2^{j}}{2^{2k} \cdot k^{2}}. \end{equation}
As before, $k$ can be chosen fairly large. Namely, the cardinality of $\calL_{e,k}$ is always bounded by $|\calL_{e,k}| \lesssim n/2^{k}$, so for $T \in \N$ one has the estimate
\begin{displaymath} \sum_{j = 0}^{T} |\calL_{e,k}| \cdot 2^{2k} \lesssim n \cdot \sum_{j = 0}^{T} 2^{k} \lesssim n \cdot 2^{T}. \end{displaymath}
The left hand side is $\ll 2^{j}$, as long as $2^{T} \ll 2^{j}/n$, and this means that $k \geq 0$ can be found satisfying \eqref{form22} and
\begin{equation}\label{form23} 2^{k} \gtrsim \frac{2^{j}}{n} \geq Cn^{1 - s} \cdot \log^{12} n, \end{equation}
where the latter inequality follows from \eqref{form21}. Now, this $k$ depends on $e$, of course, but by pigeonholing once more -- and observing that $2^{k} \lesssim n$ -- one can find a subset $E'_{j} \subset E_{j}$ of cardinality
\begin{displaymath} \frac{|E_{j}|}{\log n} \lesssim |E_{j}'| \leq |E_{j}| \end{displaymath}
such that the same $k$ works for all $e \in E_{j}'$. Finally, for all $e \in E_{j}'$, remove some lines from $\calL_{e,k}$ until $\lesssim 2^{j - 2k}$ lines remain; this is possible by \eqref{form22}.

The proof is completed by appealing to the Szemer\'edi-Trotter theorem \cite{ST}, which asserts that the incidences $I(\calL,P) = \{(\ell,p) : \ell \in \calL, \: p \in \ell \cap P\}$ between a family of lines $\calL$ and a family of points $P$ satisfies
\begin{displaymath} |I(\calL,P)| \lesssim |\calL|^{2/3}|P|^{2/3} + |\calL| + |P|. \end{displaymath}
Here, let $\calL$ be the family of all the lines in all the (possibly reduced) collections $\calL_{e,k}$, for $e \in E_{j}'$. By \eqref{form20} and \eqref{form22}, the cardinality of $\calL$ is bounded from below and above as follows:
\begin{displaymath} \frac{n^{1 + s}}{2^{2k} \cdot \log^{4} n} \lesssim \frac{n^{1 + s}}{2^{2k} \cdot (jk)^{2}} \lesssim |\calL| \lesssim \frac{n^{1 + s}}{2^{2k}}. \end{displaymath}
Each line in $\calL$ contains $\sim 2^{k}$ points of $P$, so the number of incidences $I(\calL,P)$ between $\calL$ and $P$ is $\gtrsim n^{1 + s}/(2^{k} \cdot \log^{4} n)$. Plugging this into the Szemer\'edi-Trotter bound yields
\begin{displaymath} \frac{n^{1 + s}}{2^{k} \cdot \log^{4} n} \lesssim \left(\frac{n^{1 + s}}{2^{2k}}\right)^{2/3}\cdot n^{2/3} + \frac{n^{1 + s}}{2^{2k}} + n = \frac{n^{4/3 + 2s/3}}{2^{4k/3}} + \frac{n^{1 + s}}{2^{2k}} + n. \end{displaymath}

There are three possible cases, according to which one of the three terms on the right dominates. The first term cannot do this, because the resulting inequality combined with \eqref{form23} -- for large enough $C$ -- would lead to a contradiction:
\begin{displaymath} Cn^{1 - s} \cdot \log^{12} n \lesssim 2^{k} \lesssim n^{1 - s} \cdot \log^{12} n. \end{displaymath}
The same is true of the second term, for the same reason. So, the third term dominates, and this gives
\begin{displaymath} 2^{k} \gtrsim \frac{n^{s}}{\log^{4} n}. \end{displaymath}
The proof is complete, because the lines on $\calL$ contain $\gtrsim 2^{k}$ points. \end{proof}

\end{document}